\documentclass[a4 paper]{amsart}

\usepackage{listings}
\usepackage{lscape}
\usepackage{graphicx}	% Use pdf, png, jpg, or eps§ with pdflatex; use eps in DVI mode
\usepackage[utf8]{inputenc}								% TeX will automatically convert eps --> pdf in pdflatex		
\usepackage{amssymb, amsthm, amsmath, enumitem, textcomp, theoremref, verbatim, fontawesome, hyperref, abstract}
\usepackage{mathrsfs}
\usepackage[dvipsnames]{xcolor}
\usepackage{tikz}
\usetikzlibrary{shapes,shapes.geometric, arrows, automata}
\usetikzlibrary{decorations.pathreplacing,shapes.multipart}
\makeatletter
\DeclareRobustCommand{\rvdots}{%
	\vbox{
		\baselineskip4\p@\lineskiplimit\z@
		\kern-\p@
		\hbox{.}\hbox{.}\hbox{.}
}}
\makeatother

\renewcommand{\epsilon}{\varepsilon}

\usepackage[utf8]{inputenc}

\usepackage{thm-restate}
\usepackage{enumitem}

\usepackage{hyperref}
\newtheorem{thm}{Theorem}[section]

\newcommand{\defend}{$\clubsuit$}
\newtheorem{definition}[thm]{Definition}
\newenvironment{defin}{\begin{definition}}{\hspace*{\fill} \defend \end{definition}}
\newtheorem{question}[thm]{Question}
\newtheorem{rem}[thm]{Remark}\newtheorem{lem}[thm]{Lemma}
\newtheorem{cor}[thm]{Corollary}
\newtheorem{remark}[thm]{Remark}

\theoremstyle{definition}

\usepackage{comment}

\newcommand{\defeq}{\mathrel{:=}}
\newcommand{\seteq}{\defeq}
\newcommand{\lA}{\mathcal{L}(\Aaut)}

\newcommand{\set}[2]{\left\{\,#1\mid#2\,\right\}}

\newcommand{\N}{\mathbb{N}}
\newcommand{\Z}{\mathbb{Z}}

\renewcommand{\leq}{\leqslant}
\renewcommand{\geq}{\geqslant}

\newcommand{\Aaut}{\mathcal{A}}

\newcommand{\makeset}[2]{\left\lbrace #1 \;\middle|\;
  \begin{tabular}{@{}l@{}}
    #2
   \end{tabular}
  \right\rbrace}

\usepackage[T1]{fontenc}
\usepackage{thm-restate}
\usepackage{amssymb}
\usepackage{amsmath,stackengine}
\usepackage[
  hmarginratio={1:1},     % equal left and right margins
  vmarginratio={1:1},     % equal top and bottom margins
  textwidth=400pt,        % new text width
  heightrounded,          % always useful
]{geometry}
\usepackage{xcolor}        %lots of colors

\usepackage{hyperref} 
\hypersetup{colorlinks=true, linkcolor={red!50!black}, citecolor={green!50!black}}

\renewcommand{\leq}{\leqslant}
\renewcommand{\geq}{\geqslant}

\thanks{{ \hspace{-.2in}\textit{MSC} (2020): 20F10; 20M35; 20F05; 03D05; 03D40.}
For the purpose of open access, the authors have applied a CC BY public copyright licence to any accepted manuscript version arising.
}

\newcommand{\lclasses}{\mathcal{C}}

\title{On epiC groups over language class C}
\author{Raad Al Kohli, Collin Bleak, Luna Elliott}

\stackMath

\renewcommand{\emph}{\textbf}

\begin{document}
\maketitle

\begin{abstract}We introduce a new framework linking group theory and formal language theory which generalizes a number of ways these topics have been linked in the past. For a language class C in the Chomsky hierarchy, we say a group is epiC if it admits a language $L$ over a finite (monoidal) generating set $X \subseteq G$ in the class C such that the image of L under the evaluation map is $G \setminus \{1_G\}$. We provide some examples of epiC groups and prove that the property of being epiC is not dependent on the generating set chosen. We also prove for language class C in the Chomsky hierarchy that the class epiC is closed under passage to finite index subgroups and finite index overgroups, taking extensions, and taking graph products of finitely many epiC groups.  We also provide a characterization of the property of having solvable word problem within the framework of epiC groups.  Finally, we show two theorems that stand at counterpoint: There are uncountably many pairwise non-isomorphic groups which fail to be in the broadest class epiRE, but at the same time, there are uncountably many pairwise non-isomorphic groups which are in the thinnest class epiRegular.
\end{abstract}
\tableofcontents

\section{Introduction}
We denote by $\lclasses$ the following set 
\[\lclasses\seteq\{\operatorname{Regular}, \textrm{CF}, \textrm{CS}, \textrm{RE}\}.\]
If C $\in \lclasses$, then C represents a standard formal language class from the Chomsky hierarchy \cite{Chomsky56, ChomskySchutzenberger}, with CF representing \emph{Context-Free}, CS representing \emph{Context Sensitive}, and RE representing \emph{Recursively Enumerable}.

In this article, for each C $\in \lclasses$, we introduce the class of groups epiC.  \emph{A finitely generated group $G=\langle X \rangle$
is epiC} if there is a language $L_G\subseteq X^+$ in class $C$ so that the evaluation map $eval_G:L_C\to G$ has image equaling $G\setminus\{1_G\}$.

The classes of groups epiC nest in the obvious fashion, induced by the nesting of language classes 
\[
Regular\subsetneq CF\subsetneq CS \subsetneq RE.
\]
By definition, the class epiRegular is a natural broadening of the class of groups with rational cross section, a notion which can be traced back to Eilenberg and Sch\"{u}tzenberger from \cite{EilenbergSchutzenberger1969} and which was explored in depth by Gilman and also Bodart  \cite{Gilman1987, bodartrational}.  As our results show, epiRegular groups also contain the autostackable groups of \cite{Autostackable} (and hence, also the asynchronously automatic groups and particular groups of topical interest such as Thompson's group $F$ \cite{autostackableF}). Finally, those groups which admit a regular language of geodesics representing all group elements (see \cite{Geodesic}) are also epiRegular, since the identity will be represented by the empty string which can be removed.  Thus, e.g., extra-large type Artin groups and Garside groups are epiRegular.
It follows that each of the classes epiC can be thought of as a broadening of these previously studied classes.

Note as well that it is natural to think of epiRegular as a generalisation of the class of groups with coword problem being a regular language (the finite groups by Anisimov's Theorem \cite{Anisimov}), epiCF as a generalisation of the class co$\mathcal{CF}$ of groups (see \cite{HRRTcocf,LehnertSchweitzer,BleakMatucciLehnert}) and epiCS as a generalisation of Holt and R\"{o}ver's co-indexed groups \cite{HoltRovercoindexed}. 
 The class of epiRegular groups was initially proposed by James Belk as an interesting class of groups to consider, after some analysis of a construction of  Ville Salo appearing in a preprint of \cite{salo2021graph}.  That conversation led to an investigation of epiRegular groups in the first author's dissertation \cite{AlKohliDiss}, which is the foundation of the work described here.

\subsection{Results}
\label{subsec:Results}

Here we present our main results.  

The first result is that for each language class C, the property of being epiC is a property of the groups themselves, and not of the choices of finite generating sets used to demonstrate membership in the class epiC.
\setcounter{section}{2}

\begin{restatable}{thmLet}{thmPropGrps}\label{thm:PropGrps} Suppose C\,\,\(\in\{Regular, CF, CS, RE\}\) and $G$ is a group with finite subsets $X$ and $Y$ where both $X$ and $Y$ generate $G$ as a monoid.  If $G$ is epiC with respect to $X$ then $G$ is epiC with respect to $Y$.
\end{restatable}
Our next result is that even the smallest class, epiRegular, is a fairly broad class of groups.

\begin{restatable}{thmLet}{thmContainment} \label{thm:Containments}
 The class of epiRegular groups contains:
\begin{enumerate}
    \item all Baumslag-Solitar groups $BS(m,n)$,
    \item all automatic groups, and
    \item all autostackable groups.
\end{enumerate}
\end{restatable}

The first two points above are already known through the theory of groups with rational cross section (see \cite{bodartrational,WordProcessingInGroups}).  Guba and Sapir show that Thompson's group $F$ has rational cross section \cite{GubaDehnF}, so $F$ is certainly epiRegular.  

Meanwhile, Corwin, et al. show $F$ is autostackable in \cite{autostackableF}.  We prove autostackable groups have rational cross section in Theorem \ref{thm:autostackable}, and hence they are epiRegular as well.
This result about autostackable groups is apparently new. 
However, we are not highlighting this result here as it is not too surprising in the context of  the existing work of Brittenham, et al., and of Bodart (see \cite{bodartrational,Autostackable}).
Recall the class $co\mathcal{CF}$ of groups with context-free coword problem introduced in \cite{HRRTcocf}.   Theorem \ref{thm:Containments} implies that the Baumslag-Solitar groups $BS(m,n)$ are in epiCF, but it is shown in \cite{HRRTcocf} that $BS(1,2)$ is not a $co\mathcal{CF}$ group.

\begin{cor}
    There are groups in the class epiRegular (and thus in epiCF) that are not in the class $co\mathcal{CF}$.
    \end{cor}

Recall that epiRegular groups contain the class of groups with rational cross section, and  \cite{bodartrational} states that all Baumslag-Solitar groups and all automatic groups have rational cross sections, as a consequence of groups with rational cross-section being closed under some Bass--Serre Theory constructions. 
 Following a construction in \cite{Hermillergraphproduct} we show that the class of groups with rational cross section is closed under graph products.  This closure property for graph products holds more widely, and in particular, for each of the classes epiC.

The following theorem is an amalgamation of Theorems \ref{thm:closed_under_extenions} and \ref{thm:graphproductepiRegular}, below.

\begin{restatable}{thmLet}{thmClosureProps}\label{thm:ClosurePropGroups}
Suppose C\,\,\(\in\{Regular, CF, CS, RE\}\). The class of epiC groups is closed under graph products (of finitely many groups) and extensions.
\end{restatable}

The original public draft of this article did not show that for 
C\,\,\(\in\{Regular, CF, CS, RE\}\) the class epiC is closed under passage to finite index subgroups. 
 We are grateful that in response, Corentin Bodart sent us an argument to that effect which we include now as part of our Theorem \ref{thm:ClosureFI}. 
 Then, and before the release of the amended version of this article, Andr\'{e} Carvalho 
 and Carl-Fredrik Nyberg-Brodda independently released \cite{CFpaper} which also contains a proof of this result (closure under passage to finite index subgroups) for the classes of languages Regular, CF, and RE (and any full semi-AFL).
 
\begin{restatable}{thmLet}{thmClosureFI}\label{thm:ClosureFI}
For C\,\,\(\in\{Regular, CF, CS, RE\}\) the class epiC is closed under passage to finite index subgroups and to finite index overgroups.
\end{restatable}

There is an interplay between having a machine that accepts a representative product string for every non-trivial element of a finitely generated group $G$ and the property of $G$ having a solvable word problem.  The following is shown in Subsection \ref{subsec:solvableWP}.

\setcounter{section}{7}
\setcounter{thm}{2}
\begin{restatable}{thmLet}{thmSolvableWP}\label{thm:SolvableWP}
    Let \(G\) be a group with finite generating set \(X\).  The following are equivalent:
    \begin{itemize}

    \item \(G\) has a solvable word problem; and
    \item \(G\) admits a recursively enumerable presentation \(G\cong\langle X|R\rangle\) and is in the class epiRE.
    \end{itemize}
\end{restatable}

\setcounter{section}{1}

% We are unaware of any epiRE group which does not have a recursively enumerable presentation, even though it is plausible that such exists.

As mentioned earlier, we have 
\[
epiRegular\subseteq epiCF \subseteq epiCS \subseteq epiRE.
\]
At this time, we have not yet discovered a proof that any of these classes are different from each other.  We are also unaware of an example group which does not admit rational cross section but which is epiRegular.  A number of potential examples provided by \cite{bodartrational} include the Grigorchuk groups (as finitely generated infinite torsion groups cannot admit rational cross section) as well as the group $C_2\wr (C_2\wr \Z)$.  These groups all fail to admit rational cross sections but might be epiRegular.  Note as well that 
Grigorchuk's ``first'' group $\Gamma$ is co-indexed (see \cite{HoltRovercoindexed}),  and hence is in epiCS.  It follows that $\Gamma$ shall either separate epiRegular from groups with rational cross section, or separate epiRegular from epiCS.

In another direction, there are fast algorithms to enumerate words in a regular language, our proof that recursively presented groups in epiRE have solvable word problem reduces, in the case of finitely presented groups in the class epiRegular, to an argument that resolves the word problem for these groups in time complexity Exp(Exp(n)) for any given initial word of length n  {\it under the assumption} that group elements of $G$ appear in some natural length-based ordering.   As there are finitely presented groups in epiRE where the word problem has time complexity worse than any finite tower of exponentials \cite{Cannonito1966,Cannonito1975} our argument would then prove that $epiRegular\subsetneq epiRE$. However, the definition of epiC for any language class C does not give control on {\it when} group elements get represented in the relevant language, so we cannot use this argument to separate epiRegular from epiRE, for instance.

The reader might wonder if there are any finitely generated groups which fail to be epiRE.  In fact, there are quite a few!
\setcounter{section}{7}

\begin{restatable}{thmLet}{thmUncountable}\label{thm:uncountableNonRE}
Up to isomorphism, there are continuum many finitely generated groups which are not epiRE.  
\end{restatable}
Taking Theorem \ref{thm:uncountableNonRE} on board, one might also wonder how large (up to pairwise non-isomorphism) is the class epiRE, or with an even stricter restriction, how large is the class epiRegular?  One might expect that only finitely many, or even countably many, pairwise non-isomorphic groups can use the same regular language to demonstrate epiRegularity: the truth is far stranger.  Thus,   the following theorem stands counterpoint to Theorem \ref{thm:uncountableNonRE}.  See Subsections \ref{subsec:uncountableNonEpiRE}, and \ref{subsec:uncountableEpiRegular} for these theorems.

\begin{restatable}{thmLet}{thmUncountableEpiReg}\label{thm:unctbleEpiReg}
There are continuum many pairwise non-isomorphic epiRegular groups. 
\end{restatable}

\setcounter{section}{1}

\emph{Acknowledgements}  The authors wish to thank James Belk, Corentin Bodart, Carl-Fredrik Nyberg-Brodda, Gemma Crowe, Alex Levine, and Richard Thomas for various helpful comments and discussions around the contents of this article.

\section{Languages and groups: some formal definitions}
For C $\in \lclasses$, to describe the class epiC of groups we need a few definitions.

Let $X$ be a set.  For positive natural $k$ set $X^k$ to be the set of all $k$-tuples with coordinates from $X$.  We set 
    \[X^+ \seteq \bigcup_{k=1}^{\infty}X^k,\]
and \[X^* \seteq X^+\cup\{\varepsilon\}.\]

We view a $k$-tuple $(x_1,x_2, \ldots, x_k)$ as the string $x_1 x_2 \cdots x_k$ and $\varepsilon$ as the empty string (the empty tuple). Allowing concatenation of strings, \(X^*\) becomes \emph{the free monoid generated by \(X\)} with identity element \(\varepsilon\), as is standard.    By the phrase \emph{a formal language over alphabet $X$} we mean any subset $L\subseteq X^*$.  In this article we will only consider formal languages over finite alphabets.

Throughout, we shall never need to explore the grammars of these language classes. However, we will need to recall that languages are defined by computing machines that accept them, which we call \emph{acceptor machines}. We shall write $\lA = L$ to denote that $L$ is the language accepted by an acceptor machine $\mathcal{A}$. Now recall that a language $\lA$ is in the class C if it is  accepted by an acceptor machine $\mathcal{A}$ where C is 
\begin{itemize}
    \item \emph{Regular} if $\mathcal{A}$ is a finite state automaton (FSA),
    \item \emph{Context-free} if $\mathcal{A}$ is a pushdown automaton (PDA),
    \item \emph{Context Sensitive} if $\mathcal{A}$ is a linearly bounded automaton (LBA), or
    \item \emph{Recursively Enumberable} if $\mathcal{A}$ is a Turing machine (TM).
\end{itemize}
For more details regarding these types of acceptor machines or the language classes they accept, we refer the reader to \cite{HopcroftUllman}.

\begin{defin}[Word and coword problem]
    If \(G\) is a group and \(X\) is a monoid generating set for \(G\), then we define the \textit{word problem} and \textit{coword problem} of \(G\) with respect to \(X\) by
    \[WP(G, X):= \makeset{w\in X^*}{\(w =_G 1_G\)},\] and
    \[coWP(G, X):= X^* \backslash WP(G, X)\]
respectively.  
\end{defin}
\begin{remark} Note that if $w\in X^*$ 
 we will write $w=x_1x_2\ldots x_j$ if $w$ is precisely the string $x_1x_2\ldots x_j$, while we write $w =_G g$ (for $g \in G$) to denote that the product described by the string $w$ evaluates to $g$ in $G$ (here we are assuming that $X$ is a subset of $G$).  We extend this notation to its symmetric transitive closure so we may compare products of elements of $G$, taken as strings, to determine when such expressions evaluate to the same element of $G$ under the group product. 

 Note further that we are denoting the identity of $G$ by $1_G$, and that we may refer to elements of $WP(G,X)$ as \emph{words} and elements of $coWP(G,X)$ as \emph{cowords} when the context seems clear.
 \end{remark}
 
 %\ras{In the proof of the first theorem, we have a situation where there are two strings $u, v$ with $u =_G v$. This is not covered by the above because $v \notin G$. The above should be edited to reflect the possibility that the right hand side of $=_G$ can be a string as long as both sides evaluate to the same group element.}\cpb{XXX}\end{remark}

We are now ready to formally define the classes epiC of groups which we examine in this article. 

For a given language class C, the class of groups epiC is those groups $G$ with a monoid generating set $X$ so that the set of non-identity elements in $G$ is the image of a language $\mathcal{L}\subseteq X^*$ in C under the natural evaluation map. In more detail:

\begin{defin}
Let $G$ be a finitely generated group and \(X\subseteq G\) be finite with \(X\) generating \(G\) as a monoid. Let C be a language class from \(\lclasses\) and let \(L \subseteq X^*\) be a language over the alphabet \(X\) which is in the class C.

We say \(L\) \emph{demonstrates epiC} for \(G\) if:
\begin{itemize}
    \item $L \cap WP(G,X) = \varnothing$, and
    \item for every non-identity element $h\in G$, there exists $w_h \in coWP(G,X)$ such that $w_h \in L$ and $w_h =_G h.$
\end{itemize}
If \(\mathcal{A}\) is an acceptor machine of the appropriate type to accept a language in the class C (i.e., a finite state automaton for Regular, a pushdown automaton for CF, etc.), then we say that \(\mathcal{A}\) demonstrates epiC for \(G\) if \(\lA\) demonstrates epiC for \(G\). 

The class \emph{epiC} is the class of all groups \(H\) such that \(H\) admits a finite monoid generating set $Y$ and a language \(\mathcal{K}\subseteq Y^*\) which demonstrates epiC for \(H\).
\end{defin}

\begin{rem}
Let $G$ be a group and let C be a language class from \(\lclasses\) and suppose further that $X\subseteq G$ is a finite set which generates $G$ as a monoid.  Extending the previous definition we also say \emph{\(G\) is epiC with respect to $X$} if there is a language $\mathcal{L}\subseteq X^*$ which demonstrates epiC for $G$.  We may also say \emph{$G$ is epiC} in this case.
\end{rem}

Below are some examples of epiRegular groups. Finite groups represent an extreme case of epiRegular groups, and hence are in epiC for any of our language classes. Here are two arguments (demonstrating the flexibility of the definition of epiRegularity). 
 
 Firstly, by Anisimov's theorem \cite{Anisimov}, there exists a finite state automaton accepting exactly the entirety of the coword problem for any given set of non-trivial generators. One can also use the non-trivial elements of the group as the generators, and thus build an automaton that accepts exactly the length one words over this alphabet.
 
The group $\mathbb{Z}$ generated by the symmetric generating set \(\{a, a^{-1}\}\) is epiRegular with $\{a^{n} \mid n>0\}\cup \{(a^{-1})^{n} \mid n>0\}$ demonstrating epiRegularity. 

\section{Independence of generating sets.}

In this section we show our first claimed theorem from Subsection \ref{subsec:Results}. 
 Specifically, that being in a class epiC does not depend on the choice of (finite) generating set.

\thmPropGrps*
\begin{proof}
It is a standard result (see \cite{sipser,HopcroftUllman})  that for each language class C amongst Regular, CF, CS, and RE, the class C is closed under monoid homomorphisms which do not map letters to the empty word.  Specifically, if $X$ and $Y$ are finite sets, \(\mathcal{L}\subseteq X^*\) with \(\mathcal{L}\) in the class C, and where $\phi_*:X^*\to Y^*$ is the unique monoid homomorphism extending a function $\phi:X\to Y^+$, then the set \(\mathcal{L}\phi_*:=\{v\in Y^* \mid \exists u\in \mathcal{L}, v=u\phi_*\}\) is also a language in the class C.

Let C be a chosen language class amongst Regular, CF, CS, or RE.  Let \(G\) be a  non-trivial group with finite monoid generating set \(X\) and let \(\mathcal{L}\subseteq X^*\) which demonstrates \(G\) is epiC with respect to \(X\) (note that the trivial group is clearly epiC regardless of the generating set).  Finally, let \(Y\) be some finite monoid generating set for \(G\).

For each \(x\in X\) there is \(w_x\in Y^*\) so that \(w_x=_G x\) and as \(G\) is non-trivial we can further choose \(w_x\neq \varepsilon\). 
Let \(\phi:X\to Y^*\) be the function defined by the rules that for all \(x\in X\), we have \(x\mapsto w_x\).  Note that for all \(u\in \mathcal{L}\) we have \(u\phi_*\) has the property that $u=_G u\phi_*$.

The map \(\phi_*\) now takes \(\mathcal{L}\) to a new language \(\mathcal{K}:=\mathcal{L}\phi_*\subseteq Y^*\) which is still in the language class C and which demonstrates that \(G\) is epiC with respect to \(Y\).
\end{proof}

\begin{remark}
    As seen in the proof above, we only rely on the property of the language classes being closed under monoid homomorphisms that do not map letters to the empty word. Thus for any language class C with that property, a group $G$ is epiC with respect to one finite monoid generating set $X$ if and only if $G$ is epiC with respect to another finite monoid generating set $Y$. 

    So one may wish to explore the class of epiC groups for a language class C with the above mentioned property even if C is not in the Chomsky Hierarchy. 
\end{remark}

\section{Closure Properties}

In this section we show that each class epiC is closed under finite index overgroups, extensions and graph products. We start with finite index overgroups. 
\subsection{Finite index overgroups}
Our result is as follows. This is one half of Theorem \ref{thm:ClosureFI}.
% \ras{I am of the view that finite index overgroups should come first for two reasons. Firstly, that result was actually the one that was obtained first. Secondly, the proof is technically easier - extensions have some technicality swept under the rug through the parenthetical comments. We may disagree on whether these comments make things clearer or not, but it is certainly true that fleshing out things in detail for extensions is a slightly more work than finite index overgroups - so its nice to put the easier one first. In any event, the remark currently after extensions applies to both.}
\begin{thm} \label{thm:FIovergroups}
The class epiC is closed under finite index overgroups.    
\end{thm}
\begin{proof}
    % \ras{I do say so, with the following caveat: I don't know how to \textit{fully prove} the epiCS case because I still don't understand how that machine type works. Also this seems to be theorem 4.4 at the moment. Here's my proof:}
    
    For any language class C in $\{Regular, CF, RE\}$, suppose the group $G$ is epiC with respect to $X$ (with $\mathcal{L_G} \subseteq X^*$ demonstrating epiC for G) and $H$ is a finite index overgroup of $G$. The transversal set $T$ together with $X$ generates $H$. We may assume that $1_H \in T$, so $(T \setminus \{1_H\}) \cup X$ is a finite generating set for H. Observe that any element $h \in H$ can be written as $gt$ for some $g \in G$ and $t \in T$, with $1_H = 1_G 1_H$. Thus any non-trivial element $h$ of $H$ is either a non-trivial element of $G$, or is in a coset of $G$ with a non-trivial coset representative $t \neq 1_H$. In the latter case, $h$ may be $t$ itself and so can be represented as $1_G t$, otherwise $h$ can be represented as $g t$ for non-trivial $g$. Thus $\mathcal{L_G} \cup (T \setminus \{1_H\}) \cup \mathcal{L_G} (T \setminus \{1_H\})$ demonstrates epiC for $H$ if C is known to be closed under finite unions and concatenation. 
    
    % \ras{One can also just build the automaton that accepts this language directly}
    
    % All C in $\{Regular, CF, CS RE\}$ are known to be closed under finite unions and concatenation. \ras{I should note that according to wikipedia context sensitive languages are closed under finite unions and concatenation so epiCS is closed under passage to finite index overgroup as well by the same argument, I just don't understand how the machine works for CS so have excluded it from what I've writen. But if we accept that they have those properties then the argument goes through.}
\end{proof}
\subsection{Extensions}
This result forms one half of Theorem \ref{thm:ClosurePropGroups}.
\begin{thm}\label{thm:closed_under_extenions}
Suppose C\,\,\(\in\{Regular, CF, CS, RE\}\) and
let $N$ and $Q$ be epiC. If $E$ is an extension of $N$ by $Q$ (so that $N$ is normal in $E$) then $E$ is epiC.
\end{thm}

\begin{proof}
Suppose without meaningful loss of generality that \(Q=E/N\). 
Let $X$ be a finite monoid generating set for $N$ and $Y$ be a finite subset of \(E\) such that \(\{Ny: y\in Y\}\) is a monoid generating set for \(Q\) (and these elements are distinct). 
Note that \(X\cup Y\) is a generating set for \(E\).

Every element of $E$ can be written as a product $nq$ for some $n \in N$ and $q \in \langle Y\rangle$. Therefore elements of $E$ are either elements of $N,$ elements of \(\langle Y\rangle \backslash N\) or a product of both.

Let $\mathcal{L}_1$ be the language demonstrating $N$ is epiC over the alphabet $X$.
Let $\mathcal{L}_2$ be the language demonstrating $Q$ is epiC over the alphabet $Y$ (viewing the elements of $Y$ as their image in the quotient).
It is known (see for example \cite{HopcroftUllman}) that C is closed under concatenation and finite unions. Thus we see that $\mathcal{L}':= \mathcal{L}_1 \cup \mathcal{L}_2 \cup \mathcal{L}_1\mathcal{L}_2$ is a language demonstrating $E$ is epiC  (note that every element represented by \(\mathcal{L}_1\mathcal{L}_2\) has a non-trivial image in \(Q\)). 
\end{proof}

\begin{remark}
The key ingredients in the proofs of Theorems \ref{thm:FIovergroups} \& \ref{thm:closed_under_extenions} are that the language classes are closed under concatenations and finite unions. Thus, this result holds for other language classes as long as those properties are satisfied.
\end{remark}

\begin{cor}
    Finitely generated polycyclic groups are epiRegular, and hence so are finitely generated nilpotent groups.
\end{cor}

\begin{thm}\label{thm:closed_under_finite_index_overgroups}
     Suppose C\,\,\(\in\{Regular, CF, CS, RE\}\) and
let $H$ be epiC. If $G$ is a finite index overgroup of $H$ then $G$ is epiC. 
\end{thm}

\subsection{Finite index subgroups}
Below we prove the second half of Theorem \ref{thm:ClosureFI}.
We learnt the following proof from Corentin Bodart, who generously allowed us to use it. The ideas in this proof rely on work of Benson \cite{Benson} and Bishop \cite{BishopFIsubgroup}. 
\begin{thm}
    Suppose C\,\,\(\in\{Regular, CF, CS, RE\}\). Let $G$ be finitely generated group with finite inverse closed generating set $X$. Suppose $H$ is a finite index subgroup. If $G$ is epiC then $H$ is also epiC.
\end{thm}
\begin{proof}
     In this proof we will use \(H\backslash G\) to represent the set \(\makeset{Hg}{\(g\in G\)}\) of right cosets of \(H\) in \(G\) (not to be confused with set-minus). For all \(C\in H\backslash G\) let \(t_C\in C\) be fixed (and choose \(t_H=1\)). Let \(T=\makeset{t_C}{\(C\in H\backslash G\)}\).

     We define a labeled directed graph \(D_H\) with vertex set \(H\backslash G\) and edge set 
     \[E_H:=\makeset{(C_1,x, C_2)}{\(x\in X\) and \(C_1,C_2\in H\backslash G\) satisfy \(C_1x=C_2\)}\]
     (where the edge \((C_1,x, C_2)\) goes from \(C_1\) to \(C_2\) and has label \(x\)).

    Let \(L\subseteq X^*\) be a language in \(C\) demonstrating that \(G\) is epiC.

By a (non-trivial) path in \(D_H\), we mean a sequence of (at least 1) edges of \(D_H\) such that terminal vertex of one edge is the start of the next edge in the sequence. It is well known that the set of non-trivial paths in a finite digraph from one vertex to another is always regular.

Consider the homomorphism \(h: E_H^* \to X^*\) defined by mapping each edge to its label. It follows that the set \(P_L\) of non-trivial paths \(p\) in \(D_H\) from the vertex \(H\) to itself with \((p)h\in L\) belongs to the class \(C\) (as \(C\) is closed under inverse homomorphism and intersection with regular languages).
Note that every element of \(L\) is the image of some path (starting at \(H\)) in \(D_H\) under \(h\), in particular the elements of \(L\) which represent elements of \(H\) are precisely those in the set \((P_L)h\). 

We define $k: E_H \to G$ to be the following map \[(t_{C_1}, x, t_{C_2}) \to t_{C_1} x t_{C_2}^{-1}.\] Let $X_H$ be the image of $E_H$ under $k$. Note that from the definition of \(E_H\), we have that \(X_H\subseteq H\). Moreover if \(x_0x_1\ldots x_{k-1}\in X^*\), then there is a unique path \(e_0e_1\ldots e_{k-1}\) in \(D_H\) with labeled by \(x_0x_1\ldots x_{k-1}\) which starts at the vertex \(H\). 
Moreover \((e_0)k(e_1)k\ldots (e_{k-1})k=_G x_0x_1\ldots x_{k-1} t_{Hx_0x_1\ldots x_{k-1}}^{-1}\).

In particular, the induced map \(k:E_H^*\to X_H^{*}\) maps a non-trivial path in \(D_H\) starting at \(H\) to the same point as \(h\) whenever the path also ends at the vertex \(H\).

Thus for all \(w\in P_L\), we have \((w)k=_G (w)h\). So the set of elements of \(H\) represented by words in \((P_L)k\) is the same as the set of of elements of \(H\) represented by words in \((P_L)h\) (i.e. all the non-identity elements). This in particular implies that \(X_H\) does generate \(H\).

As the language \((P_L)k\) is in the class C (note while \(k\) is a homomorphism which may map a letter to the identity it does not map a letter to the empty word), is over a generating set for \(H\), and has the correct set of evaluations, the result follows.

\end{proof}
\subsection{Graph products of groups}
We now work to show the remaining part of Theorem \ref{thm:ClosurePropGroups}:  for each language class C, the class of epiC groups is closed under graph products. 

We introduce the definition of graph products of groups as in \cite{GraphProductdef}. 

\begin{defin}\label{graphproductdef}\label{def:graphProd}
Let $\Gamma = (V(\Gamma), E(\Gamma))$ be a finite simple graph and associate to each vertex $v\in V(\Gamma)$ a group $G_v$ with a fixed group presentation $G_v\seteq \langle X_v\mid R_v\rangle$, where we further suppose the groups $G_v$ are pairwise disjoint as sets.  We define  \emph{the graph product of groups associated with the structure $(\Gamma,\{\langle X_v\mid R_v\rangle \mid v\in V(\Gamma)\})$} as the group $G_\Gamma$ given by the presentation
\[\langle X \mid R\rangle,\]
where $X = \bigsqcup_{v \in V(\Gamma)} X_v$, and $R \subseteq X^*$ is the set of relations  \[R = \left(\bigsqcup_{v \in V(\Gamma)} R_v\right) \bigsqcup \,\{[x_v, x_w] \mid x_v \in X_v, x_w \in X_w, \{v,w\} \in E(\Gamma)\}.\] 
\end{defin}

In what follows, we will only care about graph products of groups associated to structures $(\Gamma,\{\langle X_v\mid R_v\rangle\mid v\in V(\Gamma)\})$ where each set $X_v$ is a finite monoid generating set for the corresponding vertex group $G_v$.  Note in this case that the associated group $G_\Gamma$ is finitely generated as a monoid by $X$. 
 We will be less explicit in discussing the structure $(\Gamma,\{\langle X_v\mid R_v\rangle\mid v\in V(\Gamma)\})$ below and simply refer to a graph product of groups $G_\Gamma$ associated to a graph $\Gamma$.

In particular, the graph product of groups associated to the structure $(\Gamma,\{\langle X_v\mid R_v\rangle\mid v\in V(\Gamma)\})$ is generated by all the groups $G_v$ with the relations being a union of all the relations of the groups $G_v$ together with relations that ensure $[G_v, G_w]=_{G_\Gamma}1_{G_\Gamma}$ for any $\{v,w\} \in E(\Gamma)$. Consequently, direct products of finite sets of groups and free products of finite sets of groups are both types of products that can be described as particular graph products of groups.

Before proving Theorem \ref{thm:graphproductepiRegular}, we shall define some terminology and state some results from \cite{Hermillergraphproduct} which we shall rely on in our proof. For this purpose, we follow the exposition of \cite{Hermillergraphproduct}, and redirect the reader to \cite{Hermillergraphproduct} for more details. We note that the work in \cite{Hermillergraphproduct} is based on \cite{greenthesis}.

For the remainder of this subsection we shall fix a particular structure \[(\Gamma,\{\langle X_v\mid R_v\rangle\mid v\in V(\Gamma)\})\]
and its associated graph product $G_\Gamma=\langle X\mid R\rangle$ defined as in Definition \ref{def:graphProd}.  We shall also assume that there is a total order $\preceq$ on the vertices $V(\Gamma)$ of $\Gamma$.

\begin{defin}[Local string and type]
    Let $\omega \in X^*$. A non-empty contiguous substring $\omega'$ of $\omega$ is said to be a \emph{local string} if it is written in letters from the generators of exactly one vertex group, and it is not contained in a longer contiguous substring containing $\omega'$ that is written in letters of the generators of exactly one vertex group. The \emph{type of a local string} is the corresponding vertex.
\end{defin}
Since any string in $X^*$ is a product of local strings we can define the \emph{type of a string}.
\begin{defin}[Type of a string]\label{def:typeString}
    Let $\omega  \in X^*$ be a string. The concatenation of the types of the local strings of $\omega$ taken in order is the \emph{type} of $\omega$.
\end{defin}
For example, suppose $\omega\in X^*$ factors as the concatenation of three non-empty local strings $\omega=\omega_u \omega_v \omega_w$ where $\omega_u$ is a local string of type $u$, $\omega_v$ is a local string of type $v$, and $\omega_w$ is a local string of type $w$, then the type of $\omega$ is $uvw$.

Note also that in Definition \ref{def:typeString}  we are allowing the empty string as a type (it is the type of the empty string from $X^*$).

\begin{defin}[Global Length]
    The \emph{global length} of a string is the number of local strings it contains.
\end{defin}

That is, for $\omega\in X^*$ the global length of $\omega$ is the actual length of its type as a string over the alphabet of vertices of $\Gamma$.

We may \emph{shuffle} local string $\omega_u$ and $\omega_v$ in a string if $\omega_u$ and $\omega_v$ are adjacent in that string, and $\{u,v\} \in E(\Gamma)$. That is, we may use the commuting relations defined by $[G_u, G_v]=_{G_\Gamma} 1_{G_{\Gamma}}$ to obtain an equivalent string in the group when the corresponding vertices are adjacent. If in the process of shuffling, two local strings of the same type become adjacent then they amalgamate into a single local string. If it is equal to the identity of the corresponding group, then it is removed.

The process we described above eventually stabilises global length. That is, it will no longer be possible to shuffle local strings of the same type to become adjacent. 

We summarize the construction above (corresponding to the first two paragraphs of page 238 of   Hermiller and Meier's \cite{Hermillergraphproduct}, this is the 9th page of the article).

\begin{lem}\label{lem:prune}
    The process outlined above of taking a string $w$ and shuffling and amalgamating local strings in iterative fashion eventually produces a string $z$ which has a smallest representative type string $t$ with respect to ShortLex ordering on the types, and this type string $t$ is unique.
\end{lem}
    
    \begin{defin} In the statement of Lemma \ref{lem:prune} we refer to $z$ as a \emph{pruning of $w$} or we say we have \emph{pruned $w$ to $z$}.  We also say that $t$ is the \emph {pruned-type of $w$}.
\end{defin}

The following is an equivalent formulation of Proposition 3.1 of \cite{Hermillergraphproduct}.
\begin{lem}[\cite{Hermillergraphproduct},Proposition 3.1] \label{typestringuniquelem} %For all $g\in G_{\Gamma}$ there is a type $T_g$ so that any pruned string representing $g$ has type $T_g$.

Each group element in $G_{\Gamma}$ uniquely determines the type of any pruned string representing it. 
\end{lem}
We note that the identity element is the only element with $\varepsilon$ as the type of the pruned string representing it (which by definition is the empty string, as well).

In \cite{Hermillergraphproduct}, the authors give a normal form for graph products of groups (assuming one already has a normal form for every vertex group) consisting of \emph{proper} strings. They do this in their Proposition 3.2.  We do not have such a normal form in general as our vertex groups need not have normal forms.  However, our construction still relies on Hermiller and Meier's property (O), which is a sub-property of their definition of a proper string.
\begin{defin}
    A string $\omega$ is said to be \emph{proper} if  satisfies the following conditions.
    \begin{enumerate}
        \item[(L)] Each local string is in its prescribed normal form.
        \item[(O)] If $\omega = \cdots \omega_{I} \cdots \cdots \omega_J \cdots$ with $I \geq J$ and $\{v_I,v_J\} \in E(\Gamma)$ then there is a local string $\omega_K$ such that $\omega = \cdots \omega_I \cdots \omega_K \cdots \omega_J \cdots$ where $\{v_K, v_J\} \notin E(\Gamma)$.  
    \end{enumerate}
\end{defin}
The authors of \cite{Hermillergraphproduct} prove in Proposition 3.2 that the set of proper strings is a set of normal forms for $G_{\Gamma}$.

In the proof of Theorem B of \cite{Hermillergraphproduct}, the authors construct a graph which they call the \emph{admissible graph}. This graph is a finite state automaton $\mathcal{A}_{AG}(\Gamma)$ accepting the proper strings of a graph product of cyclic groups of order $2$. Further, since cyclic groups of order $2$ have only one non-trivial element, the map that sends a vertex group generator to its type induces a transformation taking any proper string  precisely to its type string. Thus as in Theorem B of \cite{Hermillergraphproduct},  $\mathcal{A}_{AG}$ determines the types of strings that are allowed. That is, for every type of a proper string in a graph product of groups, the string can be read off by following a path in the admissible graph. Conversely, any path in the admissible graph reads the type of a proper string. In the proof below, we shall replace states of the admissible graph by the automata that demonstrate the vertex groups are epiC. First we shall state a few properties of the admissible graph so that our construction is clear. For the construction of the admissible graph itself, we redirect the reader to \cite{Hermillergraphproduct}.

The admissible graph of a graph $\Gamma$, $A_{AG} (\Gamma)$, constructed in \cite{Hermillergraphproduct} is a finite state automaton with the following properties:
\begin{enumerate}
    \item There is a unique initial state, this is not a final state.
    \item Every non-initial state is labelled by some vertex $v_I$ of the graph $\Gamma$. (The construction of the admissible graph does not guarantee a unique non-initial state for every vertex $v_I$.)
    \item Every edge of the admissible graph from a state labelled by $v_J \in V(\Gamma)$ to a state labelled by $v_I \in V(\Gamma)$ has as a label the generator of $G_I$. (Recall the admissible graph is based on the graph product of cyclic groups of order $2$ over $\Gamma$, hence these groups are generated by a single element.) 
    \item Every non-initial state is an accept state.
    \item There are no $\varepsilon$-transitions in the admissible graph.
    \item The language accepted by the admissible graph is the set of strings of types of proper strings of $G_\Gamma$. 
\end{enumerate}
Further, we extract the following statement from \cite{Hermillergraphproduct}. We do not prove this here as this is covered in Proposition 3.2 as well as Theorem B of \cite{Hermillergraphproduct}.

\begin{lem}\label{graphproductepiexpression}
  Let $g \in G_{\Gamma}$ be a non-trivial element of $G_{\Gamma}$. Then there is a string $\omega$ in the generators of $G_{\Gamma}$ such that $\omega =_{G_\Gamma} g$ and the type string of $\omega$ is the type string of some proper string of the graph product of cyclic groups of order $2$ over $\Gamma$. Further, every local string of $\omega$ is non-trivial in its corresponding vertex group.  
\end{lem}

We are now ready to prove the following theorem. 

In the proof below we envisage changing between machines of the same type, each to do the local calculations over each vertex group: specifically we replace the graph above by a large machine of the appropriate type consisting of smaller machines replacing the vertices of the graph above, and which carry out the calculations local to each of those vertex groups individually. We only allow transitions between local machines when we are currently in an accept state on a local machine and the next input is a generator from a next valid group (validity being determined by the connectivity of the graph $\Gamma$). 
 Each valid transition effectively flushes the previous memory structure (either by using a new alphabet and ignoring old alphabet letters or treating them as boundaries in the new local memory calculations (e.g., in the LBA case)), each local calculation starts if it entered the local machine at an initial state and read the one letter of input which triggered the transition to this machine.
 
\begin{thm} \label{thm:graphproductepiRegular}
Let $\Gamma$ be a finite graph as above and $C$ a language class in \\$\{Regular, CF, CS, RE\}$. For every $v_I \in V(\Gamma)$, let $G_I$ be an epiC group with finite monoid generating set $X_I$. Let $G_{\Gamma}$ be the graph product of groups $\{G_I \mid v_I \in V(\Gamma)\}$ over the graph $\Gamma$. Then $G_{\Gamma}$ is epiC.
\end{thm}
\begin{proof}
 We will construct a machine  (FSA, PDA, LBA, or TM) $\mathcal{A}$ demonstrating that $G_{\Gamma}$ is epiC.
 
 For every $v_I \in V(\Gamma)$, let $\mathcal{A}_{I}$ be a machine (FSA, PDA, LBA, or TM) demonstrating that  $G_I$ is epiC.

(We shall call these $\mathcal{A}_I$ \emph{local machines}.) We will construct $\mathcal{A}$ by ``gluing'' the $\mathcal{A}_{I}$ onto every vertex of the admissible graph $\mathcal{A}_{AG}(\Gamma)$ that is labelled by $v_I$ as follows. 

For all states $p$ and $q$ of $\mathcal{A}_{AG}(\Gamma)$ that are labeled by $v_I$ and $v_J$, respectively with an edge from $p$ to  $q$ in $\mathcal{A}_{AG}(\Gamma)$ we do the following:

Remove that edge, and insert an \(\varepsilon\) transition from each final state of $\mathcal{A}_I$ to the new state $t_{q}$ which is a transitory state for entering the machine $\mathcal{A}_J$ as if no processing has occurred. How the state \(t_q\) functions depends on the type of machine we're working with, there are four cases:
\begin{enumerate}
    \item FSA: The initial state of $\mathcal{A}_J$ shall be identified with $t_q$ (we assume by the powerset construction that $\mathcal{A}_J$ has a unique initial state) and processing will carry on as normal for this machine.
    \item PDA: Each transition to $t_q$ from the local machine $\mathcal{A}_I$ shall be an $\varepsilon$-transition that ignores the stack.  The state $t_q$ will have loop transitions that are $\varepsilon$-transitions which delete any letters on the current stack one letter at a time, until the bottom-of-stack symbol $\perp$ of $\mathcal{A}$ appears.  We shall use this same bottom-of-stack symbol for each of the local machines.  With the stack cleared we now can employ any $\varepsilon$-transition to any initial state of the local machine $\mathcal{A}_J$ (where we have one such transition for each of it's initial states). From this point, processing continues as normal for that local machine. 
    
   \item LBA: For a reference to the descriptions of LBA's which we use, the reader can consult Section 10.5 of \cite{Linz}.  Unlike the three other machine types, the machine $\mathcal{A}$ for LBA's shall have a global accept state which is different from the local accept states which appear in the local machines. The essential idea for LBA's is that calculations in an LBA take place on a tape with delineated left and right boundary symbols ("[" and "]" for discussion) which bracket the input word.  The LBA may write on this tape between the boundaries as part of its calculation.  In our context, as we are calculating a sequence of accepted (local) words, when the machine sees a letter from a generator of a different local group, that letter shall be perceived as the right relevant boundary.  Left boundaries will be updated with a new ``[" symbol replacing whatever symbol appears at the final position for allowed calculations in the previous local machine.  When a local machine achieves an accept state, it shall read an $\varepsilon$-transition loop to this state which has the impact of moving the read head to the right until it is over the right boundary, which is either a letter from the generating set of the next group, or it is the right boundary symbol ``]".  In the latter case, there is a transition to the global accept state for $\mathcal{A}$.  In the other case, there is an $\varepsilon$-transition to the appropriate $t_q$ state for the next local machine. From the state $t_q$ initial processing shall be through a collection of $\varepsilon$-transitions supporting the following.  First the read head should move left one space and write a ``[" over whatever was there so that the new processing for the machine $\mathcal{A}_J$ has a good left boundary. Then it shall transition to one of the initial states of $\mathcal{A}_J$ and processing begins as normal.

   On processing of the final local word, the ``]" appears as the final boundary, so if an accept state is achieved for the local machine, we can follow with an $\varepsilon$-transition to an accept state for the global machine $\mathcal{A}$.
  \item TM: Move to any initial state of \(\mathcal{A}_I\).
\end{enumerate}

The initial state of this new automaton $\mathcal{A}$ will be the initial state of $\mathcal{A}_{AG}(\Gamma)$. The final states consist of all the final states of every copy of $\mathcal{A}_I$ that we have glued onto the admissible graph by the above process. (If there is an edge from a sub-automaton to another, then it is from a final state of one sub-automaton to the initial state of the other.)

We note that any path in $\mathcal{A}$ projects onto a path in $\mathcal{A}_{AG}(\Gamma)$ by identifying all the edges in a sub-automaton $\mathcal{A}_I$, and edges between a final state of a sub-automaton and the initial state of another will be edges between the vertices of $\mathcal{A}_{AG}(\Gamma)$ that the sub-automata were glued onto. Therefore, any path in $\mathcal{A}$ must yield the same type string as a path in $\mathcal{A}_{AG}(\Gamma)$.

Since the initial state is not an accept state and the identity element of $G_{\Gamma}$ has type $\varepsilon$, by Lemma \ref{typestringuniquelem} the identity element of $G_\Gamma$ is not represented by any accepted string of the above automaton. By Lemma \ref{graphproductepiexpression}, we know that every non-trivial element of $G_{\Gamma}$ can be expressed as a string whose type is the type of a proper string of the graph product of cyclic groups of order 2 over $\Gamma$, with every local string $\omega_I$ being non-trivial in $G_I$. For each such local string, there is a string $\omega'_I$ in the language accepted by $\mathcal{A}_I$ such that $\omega_I =_{G_I} \omega'_I$. Thus the automaton $\mathcal{A}$ simply replaces the types of proper strings of the graph product of cyclic groups with these local strings $\omega'_I$ while preserving the type string. Therefore there is a string representing every non-trivial element of $G_{\Gamma}$ in $\mathcal{L}(\mathcal{A})$, and there is no string in $\mathcal{L}(\mathcal{A})$ representing the identity element of $G_{\Gamma}$. Therefore $G_{\Gamma}$ is epiC.
\end{proof}

This completes the proof of Theorem \ref{thm:ClosurePropGroups}. We note that it follows from Theorem \ref{thm:graphproductepiRegular} that the class of groups epiC is closed under direct and free products.

% We restate Lemma 2.4 of \cite{LevineEDT0L} with $t=1$ below as this will be used in the proof of Theorem \ref{thm:FIepiregular}.
% \begin{lem}\label{lem:alex}
% Let \(G\) be a group with finite generating set \(\Sigma\), and \(H\) be finite index in \(G\) with generating set \(\Delta\). If \(R\subseteq G\) is rational over \(\Sigma\), and \(R\subseteq H\), then \(R\) is rational over \(\Delta\).
% \end{lem}
% While Theorem \ref{thm:FIovergroups} is one half of Theorem \ref{thm:ClosureFI}, the result below is the other half.
% \begin{thm}\label{thm:FIepiregular}
%     If \(G\) is epiregular and \(H\leq G\) has finite index, then \(H\) is epiregular.
% \end{thm}
% \begin{proof}
%     Let \(X_G\) and \(X_H\) be finite generating sets for \(G\) and \(H\) respectively.

%     By Poincaré lemma we can assume that \(H\) is normal as the class epiregular is closed under extensions.

%     Let \(L\subseteq X_G^{*}\) be a language demonstrating that \(G\) is epiregular. 
% Let \(L_H\subseteq X_G^*\) consisting of all words which evaluate to an element of \(H\). 
% So an element \(g\in G\) is represented by an element of \(L_H\cap L\) if and only if \(g\in H\backslash \{1_G\}\). 

% Moreover the language \(L_H\) is regular as the cayley graph of \(G/H\) with start/termianl state \(H\) is a FSA demonstating this.

% Thus \(L_H\cap L\) is regular as well.

% Finally Lemma~\ref{lem:alex} implies there is a new regular language \(L'\) over \(X_H\) with the same evaluations as \(L_H\cap L\) as required.
% \end{proof}
\section{Autostackable groups}\label{sec:autostackable}
In this section, we shall refer to generating sets as being inverse closed as this is standard in the literature regarding autostackable groups.  Recall that a generating set $X$ is inverse closed if for all $x\in X$ we have $x^{-1}\in X$.  In particular, an inverse closed generating set generates its group as a monoid as well.
\begin{defin}
    For a finitely generated group $G$ and a finite inverse closed generating set $X$, we define $ \Gamma(G,X)$ to be the associated Cayley graph. That is the directed graph with vertex set \(G\), and directed edge set $\vec{E}$; where for each $g \in G$ and $a \in X$, there is a directed edge $e_{g,a}$  with initial vertex $g$, terminal vertex $ga$, and label $a$.
\end{defin}
\begin{defin}
    Let $G$ be a finitely generated group with a finite inverse closed generating set $X$. We say that $N \subseteq X^*$ is a set of normal forms for \(G\) if for all \(g\in G\) there is a unique string which we denote by $y_g\in N$ such that \(y_g=_Gg\).
\end{defin}

\begin{defin}
    Let $G$ be a finitely generated group with a finite inverse closed generating set $X$. Let $N$ be a set of normal forms for $G$ over $X$. Whenever we have an equality of strings $y_g a = y_{ga}$ or $y_g = y_{ga} a^{-1}$, we call the edge $e_{g,a}$ of $\Gamma(G, X)$ \emph{degenerate}.

    Let $\vec{E_{N,d}} = \vec{E}_d$ be the set of all degenerate directed edges of $\Gamma(G,X)$, and let $\vec{E}_{N,r} = \vec{E}_r := \vec{E} \setminus \vec{E}_d$; we refer to elements of $\vec{E}_r$ as \emph{recursive} edges.
\end{defin}

\begin{defin}[Padding] \label{def:tupleLang} Let $X$ be an alphabet and assume $\#$ is a symbol not in $X$. 
 Set $\widehat{X}\seteq X \cup \{\#\}$ and call the set $\widehat{X}$ the \emph{padded alphabet associated with $X$} and call $\#$ the \emph{padding symbol (for $X$)}. 
 
    Let $u,v,w \in X^*$, where $u=a_1 a_2 \cdots a_r$, $v=b_1 b_2 \cdots b_s$, and $w=c_1 c_2 \cdots c_t$ with each $a_i, b_i, c_i \in X$. 
 Set $k:=max(r,s,t)$ and $\overline{X}:= \widehat{X} \times \widehat{X} \times \widehat{X} \setminus \{(\#, \#, \#)\}$. We call $\overline{X}$ the trinary extension of $X$. We define $(u,v,w)^t \in \overline{X}^*$ to be the string 
    \[(\alpha_1, \beta_1, \gamma_1)(\alpha_2, \beta_2, \gamma_2) \cdots (\alpha_k, \beta_k, \gamma_k)\]
    where 
    \begin{enumerate}
        \item $\alpha_i=a_i$ for $1 \leq i \leq r$ and $\alpha_i = \#$ for $r < i \leq k$;
        \item $\beta_i =b_i$ for $1 \leq i \leq s$ and $\beta_i = \#$ for $s < i \leq k$;
        \item $\gamma_i = c_i$ for $1 \leq i \leq t$ and $\gamma_i = \#$ for $t < i \leq k$.
    \end{enumerate}
    We call $(u,v,w)^t$ the \emph{trinary projection corresponding to} $(u,v,w)$. 
\end{defin}
For example, for the strings $hi, bye, hello$ the trinary projection corresponding to $(hi, bye, hello)$ is 
\[(h,b,h)(i,y,e)(\#,e,l)(\#,\#,l)(\#,\#,o).\]

\begin{defin}[Synchronously Regular] 
Let $X$ be a set and let $\overline{X}$ be the trinary extension of $X$. A subset $L \subseteq (X^*)^3$ is called \emph{synchronously regular} if the set $L^t$ of trinary projections of $L$ forms a regular language over the alphabet $\overline{X}$.
\end{defin}

\begin{defin}[Autostackable]\label{def:autostackable}
    Let $G$ be a finitely generated group, and let $X$ be a finite generating set for $G$ which is inverse closed. We say $G$ is \emph{autostackable} if there is a set $N$ of normal forms for $G$ over $X$ containing the empty word $\varepsilon$, a constant $k$, and a function $\phi: N \times X \rightarrow X^*$ such that the following hold.

    \begin{enumerate}
        \item The graph of the function $\phi$
        \[graph(\phi):= \{(y_g, x, \phi(y_g,x)) \mid g\in G, x \in X\}\] is a synchronously regular language.
        \item For each $g \in G$ and $x \in X$, the word $\phi (y_g, x)$ has length at most $k$ and represents the element $x$ of $G$, and: 
            \begin{enumerate}
                \item[(2d)] If $e_{g,x} \in \vec{E}_{N,d}$, then the equality of words $\phi(y_g, x) = x$ holds.
                \item[(2r)] The transitive closure $<_{\phi}$ of the relation $<$ on $\vec{E}_{N,r}$, defined by \newline
                $e' < e_{g,x}$ whenever $e_{g,x}, e' \in \vec{E}_{N,r}$ and $e'$ is on the directed path $\phi(y_g,x)$ starting at the vertex $g$, %\newline
                is a strict well-founded partial ordering.
            \end{enumerate}
    \end{enumerate}
\end{defin}

\begin{thm} \label{thm:autostackable}
If $G$ is an autostackable group with inverse closed generating set $X$ then $G$ has rational cross-section.
\end{thm}
\begin{proof}
Let $G$, $X$, $N$ and $\phi$ be as in Definition \ref{def:autostackable}, and $\overline{X}$ be the trinary extension of $X$  (as in Definition \ref{def:tupleLang}).  Denote by $g_\phi$ the set $graph(\phi)$.

By definition, as $g_\phi$ is a synchronously regular language we have that $(g_\phi)^t$ is a regular language over the alphabet $\overline{X}$. The semigroup homomorphism which extends the projection map to the first coordinate from $\overline{X}$ takes the regular language $(g_\phi)^t$ to a regular language in $(X\cup \{\#\})^*$. We may further compose by the semigroup homomorphism from $(X\cup\{\#\})^*$ to $X^*$ which replaces all appearances of $\#$ by the empty string $\varepsilon$ to find a regular language in $X^*$. This is precisely the set $N$ of normal forms for $G$. Thus showing that $G$ has rational cross-section. 
\end{proof}
This completes the proof of Theorem \ref{thm:Containments}.
\newcommand{\Fab}{F_{\{a,b\}}}
\newcommand{\Fon}[1]{F_{#1}}
\newcommand{\insight}[1]{{\color{brown}#1}}
\newcommand{\insightred}[1]{{\color{red}#1}}

\section{Non-examples}
In this section we shall demonstrate two theorems.  The first theorem is that there is an equivalence for finitely generated groups between having solvable word problem, and having a recursively enumerable presentation and being in the class epiRE.  It follows (see Corollary \ref{Cor:FPnotepiRE}) that there are finitely presented groups that are not epiRE.  The second theorem is that there are continuum many pairwise distinct isomorphism classes of 2-generated groups which are solvable with derived length three, and which are not epiRE. 
\subsection{Word problems}\label{subsec:solvableWP}
For a set $X$ (and formal set of ``inverses'' $X^{-1}$ for elements of $X$) and $R\subseteq (X\cup X^{-1})^*$ we say a group $G$ is presented by $\langle X | R \rangle$ if the normal closure of the images of $R$ in \(F_X\) is a subgroup of \(F_X\) with quotient isomorphic to \(G\), and we write in this case $G\cong \langle X | R \rangle$. 

In this section, we shall recall that Turing machines can be used in other ways apart from accepting languages, such as enumerating and outputting sets.

Translating Higman \cite{Higman65} for the case of a finitely generated group, we say a group $G$ with finite generating set $X$ has a \emph{recursively enumerable presentation} $G \cong\langle X | R\rangle$ if there is a Turing machine which will output the family $R$ of strings over \((X\cup X^{-1})^*\).  We say $G$ has a \emph{computable presentation} $G \cong\langle X | R\rangle$ if there is a Turing machine which determines in finite time whether or not a given string in \((X\cup X^{-1})^*\) is in the set $R$. 

We shall now prove the final theorem mentioned in Subsection \ref{subsec:Results}.

\thmSolvableWP*

\begin{proof}
\((\Rightarrow):\) 

One can check for each string in $X^*$ whether or not it is a coword, and accept just the cowords. Outputting the set of words shows that $G$ admits a recursively enumerable presentation.

\((\Leftarrow):\)
 Let \(f:\N \to (X\cup X^{-1})^*\) be some recursive enumeration of a language demonstrating that \(G\) is epiRE. 
 
Let \(T\) be the Turing machine which enumerates the language $R$.
 
From $T$ we can build another Turing machine $N$ which will output $W$, the set of all the words over \((X\cup X^{-1})^*\) which evaluate as elements of the normal closure.
 
 Let \(g:\N \to W\) be the corresponding enumeration of the output of $N$.
 
    The algorithm for checking if a string \(w\in (X\cup X^{-1})^*\)  belongs to \(W\) is as follows:

    \begin{enumerate}
        \item Loop over all elements \(i\) of \(\N\)
        \begin{enumerate}
            \item Check if \((i)g=w\), if yes then \(w\) is in the word problem.
            \item Check all \(j, k\leq i\) so see if \((j)fw^{-1}=(k)g\). 
            If the answer is ever yes then we declare that \(w\) is not in the word problem. 
        \end{enumerate}
    \end{enumerate}
\end{proof}

\begin{cor} \label{Cor:FPnotepiRE}
    There are finitely presented groups which are not epiRE.
\end{cor}

\begin{cor}
    For C\,\,$\in\{Regular, CF, CS, RE\}$, every finitely presented epiC group has a solvable word problem.
\end{cor}

\subsection{Continuum many non-epiRE 2-generated groups}\label{subsec:uncountableNonEpiRE}
One might think that as there are continuum many two-generated groups with unsolvable word problem that it should be immediately clear that there are continuum many non-epiRE groups with unsolvable word problem.  However, Theorem \ref{thm:SolvableWP} only applies for groups with recursive presentations, and of course, there are only countably many such groups up to isomorphism.  Sadly, we also do not have an idea of how to specify an epiRE group which does not admit a recursively enumerable presentation.

In this subsection we closely follow a general construction due to Phillip Hall (see \cite{Hall54}), but as given in de la Harpe's \cite{delaHarpe}.  The discussion can be found in discussion points 40--43  Section III.C.  We give more details than appear in that source, and along the way we also correct a minor error in the text (in \cite{delaHarpe}, the author asserts that a set of group elements freely generate a free abelian group when they do not, but, that is not important for the overall construction). In doing so we give a slightly modified generating set and show freeness.  We will use aspects of that given construction in our proof of Theorem \ref{thm:uncountableNonRE}.

We aim to show that the cardinality of the  set of isomorphism classes of $2$-generated groups is of size continuum. 
We will actually aim for more than this statement, with an eye towards applications to epiRE groups.

The outline of what is to come is as follows:
\begin{enumerate}
\item demonstrate that the free group $\Fab$ freely generated by \(\{a, b\}\), has continuum many normal subgroups; and in particular we can find a quotient which is solvable of length 3 and whose center contains \(\oplus_{i\in \mathbb{N}} \mathbb{Z}\).
\item demonstrate that taking quotients by these normal subgroups will generate continuum many pairwise non-isomorphic groups; and moreover we can choose these normal subgroups to be in a chain.
\item We then show that any two comparable normal subgroups of \(\Fab\) give rise to quotients which cannot be shown to be epiRE by a common recursively enumerable language.
\end{enumerate}

\begin{lem}[\cite{delaHarpe}, III.C., 40]\label{lem:uncountablyManyNormalSubgroupsF2}
The set of normal subgroups of $\Fab$ is of size continuum.
\end{lem}

\begin{proof}
We follow de la Harpe's offered proof, inserting some commentary along the way.

Set $S=\set{s_i}{i\in\Z}$ as a set of generators indexed by $\Z$.  Define $R$ a set of relators as

\[
R\seteq\set{[[s_i,s_j],s_k]}{i,j,k\in\Z}\cup\set{[s_i,s_j][s_{j+k},s_{i+k}]}{i,j,k\in\Z}.
\]

We will initially be concerned with the group $\Gamma_0\seteq \langle S\mid R\rangle$. 
Set $u_j = [s_0,s_j]\in \Gamma_0$ for $j\in\Z \setminus \{0\}$, and $U\seteq \{u_i\mid i\in\Z \setminus \{0\} \}\subseteq \Gamma_0$. 

We see immediately from the relations of the form $[[s_i,s_j],s_k]$ that each element of $U$ lies in the centre of $\Gamma_0$.  Secondly, the relators of the form $[s_i,s_j][s_{j+k},s_{i+k}]$ force $[s_i,s_j]=[s_{i+k},s_{j+k}]$ for all $i, j, k\in\Z$, so each commutator of the form $[s_i,s_j]$ is an element of $U$.

In particular, for positive natural $i$ we have $[s_{-i},s_0]=[s_0,s_i]$ from which it follows that $[s_0,s_{-i}]=[s_0,s_i]^{-1}$.

\underline{Claim (a):} The set $U$ is a generating set for the commutator subgroup \([\Gamma_0, \Gamma_0]\) of \(\Gamma_0\), which is contained in the center of \(\Gamma_0\).

\newcommand{\ib}{\overline {i}\,}
\newcommand{\jb}{\overline {j}\,}
\newcommand{\kb}{\overline {k}\,}
\newcommand{\rb}{\overline {r}\,}
{\flushleft {\it Proof of claim (a)}:}\\
For notational purposes in calculations, for $i\in\Z \setminus \{0\}$, we will use $i$ to represent $s_i$ and $\ib$ to represent $s_i^{-1}$.  For a word $w\in (S\cup S^{-1})^*$, we use $\overline {w}$ for the inverse word so that $w\overline w=1_{\Gamma_0}$.

The first observation we make below is that if we have a commutator $c$ of two words each of length one, then this can be re-expressed as a commutator of positive generators from $S$, so such a commutator is in the centre of $\Gamma_0$ and in fact there is a positive integer $j$ so that $c=u_j^{\pm 1}$.

To verify this, see the following two calculations using our notational convention mentioned earlier (remembering that \(U\) is contained in the center of \(\Gamma_0\)).
\begin{align*}
    [\ib,j]&=i\jb \ib j=i\jb \ib ji\ib=i[j,i]\ib=[j,i]i\ib=[j,i]
\end{align*}

\begin{align*}
    [\ib,\jb]&=ij \ib \jb=ij \ib \jb ij\jb\ib=ij[i,j]\jb\ib=[i,j]\\
\end{align*}

One can now use standard commutator identities to express any commutator of two words from $(S\cup S^{-1})^*$ as products of conjugates of commutators involving subwords of the two words, where each commutator is shorter in length than the original. To that end, let $x,y,z\in (S\cup S^{-1})^*$ be non-trivial.  We have the standard commutator equivalences

\begin{align*}
[x,zy]&=[x,y]\cdot [x,z]^{y}
\end{align*}
and 
\begin{align*}
[zy,x]&=[z,x]^y\cdot[y,x].
\end{align*}
(where we use the convention $a^b=b^{-1}ab$ for conjugation).

Using these and our process of removing inverse generators from short commutators, we may inductively break any commutator in $\Gamma_0$ down into products of conjugates of commutators of pairs of generators.  By our relations we may ignore conjugation of commutators of pairs of generators, thus it follows that the set $U$ generates the commutator subgroup of $\Gamma_0$, which is also in the center of $\Gamma_0$ (as \(U\) is contained in the center of \(\Gamma_0\)). 

This completes the proof of Claim (a).\hfill$\diamond$

\vspace{.05 in}

\underline{Claim (b):} The commutator subgroup $[\Gamma_0,\Gamma_0]$ is free abelian\footnote{It is claimed in point (a) of the proof of the lemma in \cite{delaHarpe}, that the commutator subgroup $[\Gamma_0,\Gamma_0]<\Gamma_0$ is the free abelian group on the set $\{u_i\mid i\in\Z, i\neq 0\}$, but this is false because of the identities $[s_0,s_{-i}]=[s_0,s_i]^{-1}$.} with free generating set \(U_+:= \{u_i\mid i>0\}\).

{\flushleft {\it Proof of claim (b)}:}\\
We have already shown that this group is abelian and generated by \(U_+\). 
We show that there are no more relations by constructing an explicit group homomorphism from \(\Gamma_0\) to a group such that the elements in the image of \(U_+\) satisfy no more relations. 
Consider the set
\[(\oplus_{i\in \Z} \Z) \times (\oplus_{i\in \Z_{>0}} \Z)\]
(where \(\oplus\) means take all tuples with only finitely many non-zero coordinates) with operation
\[((a_i)_{i\in \Z}, (b_i)_{i>0}) \cdot ((c_i)_{i\in \Z}, (d_i)_{i>0})= \left(\left(a_i + c_i\right)_{i\in \Z}, \left(b_i + d_i + \sum_{\substack{k, l\\
l - k = i}} a_lc_k\right)_{i>0}\right). \]
This operation is associative as
\[(((a_{i, 0})_{i\in \Z}, (b_{i, 0})_{i>0}) \cdot ((a_{i, 1})_{i\in \Z}, (b_{i, 1})_{i>0})) \cdot ((a_{i, 2})_{i\in \Z}, (b_{i, 2})_{i>0})\]

\[= \left(\left(a_{i, 0} + a_{i, 1}\right)_{i\in \Z}, \left(b_{i, 0} + b_{i, 1} + \sum_{\substack{k, l\\
l - k = i}} a_{l, 0}a_{k, 1} \right)_{i>0}\right) \cdot ((a_{i, 2})_{i\in \Z}, (b_{i, 2})_{i>0})\]
\[= \left(\left(a_{i,0} + a_{i, 1} + a_{i, 2}\right)_{i\in \Z}, \left(b_{i, 0} + b_{i, 1} + b_{i, 2} + \sum_{\substack{k, l\\
l - k = i}} a_{l, 0}a_{k, 1}  + \sum_{\substack{k, l\\
l - k = i}} (a_{l, 0} + a_{l, 1})a_{k, 2} \right)_{i>0}\right)\]
\[= \left(\left(a_{i, 0} + a_{i, 1} + a_{i, 2}\right)_{i\in \Z}, \left(b_{i, 0} + b_{i, 1} + b_{i, 2} + \sum_{\substack{0\leq c< d\leq 2}} \sum_{\substack{k, l\\
l - k = i}} a_{l, c}a_{k, d}\right)_{i>0}\right).\]

Thus we have a semigroup which we will denote by \(\Gamma_0'\). 
It then follows immediately from the definition that \(\Gamma_0'\) is a group with the inverse operation given by: 
\[((a_i)_{i\in \Z}, (b_i)_{i>0})^{-1} = \left((-a_i)_{i\in \Z}, \left( \sum_{\substack{k, l\\
l - k = i}} a_la_k  -b_i\right)_{i>0}\right).\]
It follows from von Dyck's Theorem  that the map \(\phi\) from \(S\) to \(\Gamma_0'\) defined by
\((s_i)\phi = ((a_j)_{j\in \Z}, (0)_{i> 0})\) where \(a_j=0\) for \(j \neq i\) and \(a_i = 1\) extends to a group homomorphism over \(\Gamma_0\) if and only if the required relations \(R\) are satisfied. 
Note that for all \(i, j \in \Z\) with \(j>i\) we have

\begin{align*}
    ((s_i)\phi)^{-1}((s_j)\phi)^{-1}(s_i)\phi(s_j)\phi&= ((s_i)\phi)^{-1}((s_j)\phi)^{-1} ((\ldots  00\stackengine{2pt}
  { 1}
  {\scriptstyle i^{th}\text{ position}\rule[1pt]{2ex}{.5pt}\mkern-5mu\raisebox{1.1pt}
  {\stackon[-.1pt]{\rule{.5pt}{9.1pt}}{\uparrow}}}{U}{r}{F}{T}{S}
00\ldots 00
\stackengine{2pt}
  {1}
  {\scriptstyle\raisebox{1.1pt}{\stackon[-.1pt]{\rule{.5pt}{9.1pt}}{\uparrow}}
  \mkern-5mu\rule[1pt]{2ex}{.5pt} j^{th}\text{ position}}{U}{l}{F}{T}{S}00 \ldots ),(00\ldots   ))\\
  &= ((s_i)\phi)^{-1} ((\ldots  00\stackengine{2pt}
  { 1}
  {\scriptstyle i^{th}\text{ position}\rule[1pt]{2ex}{.5pt}\mkern-5mu\raisebox{1.1pt}
  {\stackon[-.1pt]{\rule{.5pt}{9.1pt}}{\uparrow}}}{U}{r}{F}{T}{S}
00 \ldots ),(00\ldots  00(-\stackengine{2pt}
  {1}
  {\scriptstyle\raisebox{1.1pt}{\stackon[-.1pt]{\rule{.5pt}{9.1pt}}{\uparrow}}
  \mkern-5mu\rule[1pt]{2ex}{.5pt}(j-i)^{th}\text{ position}}{U}{l}{F}{T}{S})00\ldots ))\\
    &= ((\ldots  00 \ldots ),(00\ldots  00(-\stackengine{2pt}
  {1}
  {\scriptstyle\raisebox{1.1pt}{\stackon[-.1pt]{\rule{.5pt}{9.1pt}}{\uparrow}}
  \mkern-5mu\rule[1pt]{2ex}{.5pt} (j-i)^{th}\text{ position}}{U}{l}{F}{T}{S})00\ldots )),
\end{align*}

and for \(j<i\) we have

\begin{align*}
    ((s_i)\phi)^{-1}((s_j)\phi)^{-1}(s_i)\phi(s_j)\phi&= ((s_i)\phi)^{-1}((s_j)\phi)^{-1} ((\ldots  00\stackengine{2pt}
  { 1}
  {\scriptstyle i^{th}\text{ position}\rule[1pt]{2ex}{.5pt}\mkern-5mu\raisebox{1.1pt}
  {\stackon[-.1pt]{\rule{.5pt}{9.1pt}}{\uparrow}}}{U}{r}{F}{T}{S}
00\ldots 00
\stackengine{2pt}
  {1}
  {\scriptstyle\raisebox{1.1pt}{\stackon[-.1pt]{\rule{.5pt}{9.1pt}}{\uparrow}}
  \mkern-5mu\rule[1pt]{2ex}{.5pt} j^{th}\text{ position}}{U}{l}{F}{T}{S}00 \ldots ),(00\ldots  00\stackengine{2pt}
  {1}
  {\scriptstyle\raisebox{1.1pt}{\stackon[-.1pt]{\rule{.5pt}{9.1pt}}{\uparrow}}
  \mkern-5mu\rule[1pt]{2ex}{.5pt} (i-j)^{th}\text{ position}}{U}{l}{F}{T}{S}000\ldots ))\\
 &= ((s_i)\phi)^{-1}((\ldots  00\stackengine{2pt}
  { 1}
  {\scriptstyle i^{th}\text{ position}\rule[1pt]{2ex}{.5pt}\mkern-5mu\raisebox{1.1pt}
  {\stackon[-.1pt]{\rule{.5pt}{9.1pt}}{\uparrow}}}{U}{r}{F}{T}{S}
00 \ldots ),(00\ldots  00\stackengine{2pt}
  {1}
  {\scriptstyle\raisebox{1.1pt}{\stackon[-.1pt]{\rule{.5pt}{9.1pt}}{\uparrow}}
  \mkern-5mu\rule[1pt]{2ex}{.5pt}(i-j)^{th}\text{ position}}{U}{l}{F}{T}{S}000\ldots ))\\
   &= ((\ldots  00 \ldots ),(00\ldots  00\stackengine{2pt}
  {1}
  {\scriptstyle\raisebox{1.1pt}{\stackon[-.1pt]{\rule{.5pt}{9.1pt}}{\uparrow}}
  \mkern-5mu\rule[1pt]{2ex}{.5pt}(i-j)^{th}\text{ position}}{U}{l}{F}{T}{S}000\ldots )).\\
\end{align*}
In both of the above cases the only non-zero entry of \([(s_i)\phi, (s_j)\phi]\) is in the \(|i-j|\) coordinate of the second coordinate and its value is determined by which is bigger. Thus the required relations defined by the set \(R\) hold.
Moreover, this map sends the set \(U\) to a free basis for the free abelian group \(\{(0)_{i\in \Z}\} \times \oplus_{i>0} \Z\leq \Gamma_0'\), hence \(U\) does indeed generate a free abelian group.

This completes the proof of Claim (b).\hfill$\diamond$

\vspace{.05 in}

Now returning to the proof in \cite{delaHarpe}, let $\phi:\Gamma_0\to\Gamma_0$ be the automorphism which maps $s_i$ to $s_{i+1}$ for all indices $i$.  Because of the second part of the relations $R$, the automorphism $\phi$ fixes the set $U$ pointwise, and so, fixes the commutator subgroup $[\Gamma_0,\Gamma_0]$ pointwise.  Also observe that the action of $\langle\phi\rangle$ on the set $S$ is transitive.

We define a group
\[
\Gamma \seteq \langle S\cup \{t\}\mid R\cup \{s_i^ts_{i+1}^{-1} \mid i\in \mathbb{Z}\}\rangle.
\]
 It follows from the construction that $\Gamma\cong\Gamma_0\rtimes_{\phi}\Z$, with conjugation by $t$ acting as the automorphism $\phi$, and $t$ generating the $\Z$ factor in the quotient of $\Gamma$ by $\Gamma_0$.  We consider in the background the homomorphism $\psi:\Fab\to\Gamma$ determined by $a\mapsto s_0$ and $b\mapsto t$.  
Distinct normal subgroups of $\Gamma$ are then the images of distinct normal subgroups of $\Fab$, so we reduce the problem to finding continuum many normal subgroups of $\Gamma$. 
\footnote{From the first claim $\Gamma_0$ is nilpotent of class 2.  Further, it is immediate that $\Gamma$ is solvable with derived length three, but not nilpotent as 
\([s_i,t]=s_i^{-1}s_{i+1}\), \([s_i^{-1}s_{i+1},t]=s_{i+1}^{-1}s_is_{i+1}^{-1}s_{i+2}\)
and inductively, if we keep commutating the result of the previous calculation by $t$, the rightmost term of the resulting product will be a sole occurrence of $s_{i+k}$ (where we have carried out $k$ commutations).}

We now complete the proof of the lemma.  Let $X$ be any subset of $\N\backslash\{0\}$,\footnote{\cite{delaHarpe} uses $\Z\backslash\{0\}$ here, not recognizing e.g. that $u_{-3}=u_3^{-1}$.}
and let $N_X$ denote the subgroup of $[\Gamma_0,\Gamma_0]$ generated by $U_X\seteq\{u_x\mid x\in X\}$.  As $[\Gamma_0,\Gamma_0]$ is central in $\Gamma$, it follows that each $N_X$ is normal, and that $\Gamma$ therefore admits continuum many normal subgroups.

\end{proof}

There is now an observation in \cite{delaHarpe} that Lemma \ref{lem:uncountablyManyNormalSubgroupsF2} in effect shows $\Fab$ admits continuum many normal subgroups containing $D^3\Fab$ (the third derived subgroup of $\Fab$).  This is a result of Hall from 1954 (see \cite{Hall54}) where he actually shows that there is a group $G$ satisfying $[D^2G,G]=1$ so that there are continuum many normal subgroups of a group $G$.  Hall in that paper also shows that the number of normal subgroups of $\Fab$ containing $D^2\Fab$ is countable.

A version of the lemma in Point 42 of \cite{delaHarpe} is as follows:
\begin{lem}[\cite{delaHarpe}, III.C., 42]\label{lem:uncountablyManyNonIsoQuotientsF2}
For a finitely-generated group $\Gamma$, TFAE:
\begin{enumerate}
    \item $\Gamma$ has continuum many normal subgroups; and
    \item $\Gamma$ has continuum many pairwise non-isomorphic quotients.
\end{enumerate}
\end{lem}
\begin{proof}
That (ii) implies (i) is immediate just by considering the kernels of the quotient maps.  Thus we need only show that (i) implies (ii).

Let us now assume $\Gamma$ is generated by $k$ generators for some natural $k$, has continuum many normal subgroups, and only countably many quotients up to isomorphism.  In this case there must be a countable group $Q$ and continuum many distinct quotient maps from $\Gamma$ to $Q$.  However, this is impossible since  any homomorphism from $\Gamma$ to $Q$ is determined by where the generators of $\Gamma$ are sent, and there are only countably many sequences of length $k$ in $Q$.
\end{proof}

We have seen from Lemma \ref{lem:uncountablyManyNormalSubgroupsF2} and Lemma \ref{lem:uncountablyManyNonIsoQuotientsF2} that the free group $\Fab$ admits continuum many non-isomorphic quotients.  In fact, the group $\Gamma$, which is a solvable quotient of $\Fab$ with derived length three, has continuum many non-isomorphic quotients.  We are now in a position to prove Theorem \ref{thm:uncountableNonRE} mentioned in Subsection \ref{subsec:Results}: specifically, we argue that $\Gamma$ has many quotients which are not epiRE (and in fact we distinguish continuum many such isomorphism classes). 
\thmUncountable*
\begin{proof}
The first point to observe is that the set of subsets of $\N\backslash\{0\}$ admits chains of size continuum under the containment relation.
This can be seen for example by noting that the set of downwards closed subsets of \((\mathbb{Q}, \leq)\) is of size continuum and totally ordered by \(\subseteq\) (\(\mathbb{Q}\) has the same cardinality as $\N\backslash\{0\}$).

Our main result now follows easily from our construction. For every subset \(X\subseteq \N\backslash\{0\}\), let \(N_X\) denote the normal subgroup of \(\Gamma\) generated by the set \(\{u_i\mid i \in X\}\). 
For each normal subgroup $N_X$, the quotient group $Q_X=\Gamma/N_X$  is determined, and for each $Q_X$ let us re-use $Y\seteq \{s,s^{-1},t, t^{-1}\}$ as symbols for the induced generating set of $Q_X$ (that is, e.g., $s\in Q_X$ is formally the coset $sN_X$ of $N_X$ in $\Gamma$, and the four cosets induced from $Y$ under the natural quotient map form a generating set for $Q_X$).

Suppose for a contradiction that only countably many of the quotients \(Q_X\) are not epiRE up to isomorphism. By the same argument given in Lemma \ref{lem:uncountablyManyNonIsoQuotientsF2}, there are only countably many sets \(X\) (up to equality) for which \(Q_X\) is not epiRE.
From this point on we only consider the subsets \(X\subseteq \mathbb{N}\backslash\{0\}\) which are contained in some fixed continuum sized chain of subsets of \(\mathbb{N}\backslash\{0\}\). We may also choose this chain such that all of the corresponding groups \(Q_X\) are epiRE (by removing the countably many problematic sets).

As there are (up to relabelling) only countably many Turing machines which use the alphabet $Y$, we must have that one of these machines must determine a language $\mathcal{L}_{Y}$ which demonstrates that two distinct groups $Q_X$ and $Q_{X'}$ arising in our construction are epiRE and we may assume $X'\subsetneq X$.

As the sets $X$ and $X'$ are distinct, there is some $u_i$ a generator of $N_X$ which is not in the group $N_{X'}$.  As \(\mathcal{L}_Y\) demonstrates that \(Q_{X'}\) is epiRE, there is a word $w\in \mathcal{L}_Y$ representing the non-trivial image of $u_i$ in $Q_{X'}$.
 As \(\mathcal{L}_Y\) demonstrates that \(Q_{X}\) is epiRE and \(w\in \mathcal{L}_Y\), the word \(w\) represents a non-trivial element of $Q_{X}$.
But as \(w\) represents $u_iN_{X'}$ in \(Q_{X'}\) and \(u_iN_{X'}\subseteq u_iN_{X}\), \(w\) represents the element \(u_iN_{X}\) of \(Q_X\).
This is a contradiction as \(u_iN_{X}=N_X\) is the identity of \(Q_X\).
\end{proof}
\subsection{Continuum many epiregular groups}\label{subsec:uncountableEpiRegular}
\thmUncountableEpiReg*

\begin{proof}
    Consider the groups
\[G_{I}:=\left\langle a, y, z \mid a^2=z^2=[z,a]=[z,t]=1, [t^nat^{-n},a]=\begin{cases}
    1& n\in I \\
    z& n\not\in I 
\end{cases}
\right\rangle\]
where \(I\subset \{1, 2, \ldots\}\).
By the comments in the proof of Corollary 1.6 on page 8 of \cite{FFG}, the group \(G_I\) is a central extension of the cyclic group of order 2 by the lamplighter group and \(G_I\cong G_J\iff I=J\). It follows that there are continuum many epiregular groups up to isomorphism. 
\end{proof}
We obtain the following corollaries
\begin{cor}
Suppose C\,\,\(\in\{Regular, CF, CS, RE\}\).
We obtain the following surprising corollaries:
\begin{itemize}
    \item There is a language \(L\) in \(C\) which alone can demonstrate the epiCness of uncountably many groups up to isomorphism.
    \item  There is an epiC group with unsolvable word problem.
    \item There is a bijection between the isomorphism classes of f.g. epiC groups and the isomorphism classes of f.g. not epiC groups.
\end{itemize}
\end{cor}

\begin{remark} As the class of groups with rational cross section is closed under extensions and includes the lamplighter group \cite{bodartrational} the proof above actually shows that there are continuum many groups with rational cross section .
\end{remark}

\section{Questions}\label{sec:questions}
In this section we will explore some questions which are natural given the explorations above.

Having introduced the hierarchy of epiC groups, we wish to understand the relationship between these classes of groups and whether the classes can be separated in a way that is analogous to the Chomsky hierarchy. More precisely we have the following question.

\begin{question}
    Let C$_1$ and C$_2$ be distinct language classes in the Chomsky hierarchy such that C$_1$ is a subclass of C$_2$. Does there exist a group in epiC$_2$ but not epiC$_1$?

\end{question} 
If the answer to the previous question is affirmative, then we are also interested in the following related question.
\begin{question} 
    Let C$_1$ and C$_2$ be distinct language classes such that C$_1$ is a subclass of C$_2$. Does there exist a reasonable group theoretic property $\mathcal{P}$ such that a group $G \in$epi$C_1$ is equivalent to $G$ has $\mathcal{P}$ and $G \in$epiC$_2$?
\end{question}
It is known that finitely generated infinite torsion groups do not admit rational cross sections \cite{Gilman1987,Ghysbook}, so there is a fair chance that these groups are not epiRegular.  However, Grigorchuk's first group $\Gamma$ is known to be co-indexed \cite{HoltRovercoindexed}, and so $\Gamma\in$ epiCS.  Thus, there is a chance that Grigorchuk's group can be used to separate the classes epiRegular and epiCS.  Specifically, these considerations lead us to ask the following questions.

\begin{question} \label{q: epiReg vs rat}
Does there exist an epiRegular group which does not admit a rational cross section?
\end{question}

We would be happy even with stronger versions of the above question, for instance if we assume the group is recursively presented or finitely presented.

\begin{question}
Is Grigorchuk's first group $\Gamma$ in epiRegular?
\end{question}
If Grigorchuk's group is epiRegular then, being epiRegular is not equivalent to admitting a rational cross section. On the other hand, if Grigorchuk's group is not epiRegular then, being epiRegular is not equivalent to being epiCS.

Further, it is shown in \cite{bodartrational} that $C_2 \wr (C_2 \wr \mathbb{Z})$ does not have rational cross-section. We currently do not know whether or not $C_2 \wr (C_2 \wr \mathbb{Z})$ is epiRegular. Understanding this may also prove useful in answering Question \ref{q: epiReg vs rat}. 

\begin{question} 
    Is $C_2 \wr (C_2 \wr \mathbb{Z})$, the standard restricted wreath product of $C_2$ with the Lamplighter group, an epiRegular group? 
\end{question}

\bibliographystyle{amsplain}
\bibliography{bibfiles.bib}
\vspace{.2 in}

\flushleft{Raad Al Kohli, 
\texttt{rsasak@st-andrews.ac.uk}\\
School of Mathematics \& Statistics, University of St Andrews, St Andrews, UK\\[5pt]
Collin Bleak, 
\texttt{collin.bleak@st-andrews.ac.uk}\\
School of Mathematics \& Statistics, University of St Andrews, St Andrews, UK\\[5pt]
Luna Elliott,
\texttt{luna.elliott@manchester.ac.uk}\\
Department of Mathematics, The University of Manchester, Manchester, UK}
\end{document}